\documentclass[11pt,a4paper]{elsart}

\usepackage[a4paper, total={6in, 8in}]{geometry}

\usepackage{amssymb}
\usepackage[english,francais]{babel}
\usepackage{amsmath, amsfonts, epsfig, xspace, bm, mathrsfs}
\usepackage{url}
\usepackage{xcolor}

\newtheorem{e-proposition}[theorem]{Proposition}

\newtheorem{e-definition}[theorem]{Definition\rm}
\newtheorem{remark}{\it Remark\/}

\newtheorem{proof}{Proof}
\newtheorem{theoreme}{Th\'eor\`eme}[section]

\newtheorem{proposition}[theoreme]{Proposition}

\def\og{\leavevmode\raise.3ex\hbox{$\scriptscriptstyle\langle\!\langle$~}}
\def\fg{\leavevmode\raise.3ex\hbox{~$\!\scriptscriptstyle\,\rangle\!\rangle$}}

\begin{document}
\centerline{}
\begin{frontmatter}


\selectlanguage{english}
\title{A mixed EIM-SVD tensor decomposition for bivariate functions}


\selectlanguage{english}
\author[authorlabel1]{Florian De Vuyst},
\ead{florian.de-vuyst@utc.fr}
\author[authorlabel2]{Asma Toumi}
\ead{asma.toumi@ens-paris-saclay.fr}

\address[authorlabel1]{LMAC, Sorbonne Universit\'es, Universit\'e de Technologie de Compi\`egne, 60200 Compi\`egne cedex, France}
\address[authorlabel2]{CMLA, 
ENS Paris-Saclay, CNRS, 61 avenue de Pr\'esident Wilson 94235 Cachan cedex, France.}


\medskip

\begin{abstract}
\selectlanguage{english}

In this paper we present a mixed EIM-SVD tensor decomposition for bivariate functions. This method is composed, as its name suggests, of two main steps. The first one, provides an approximate representation of a function $f$ in separate form by the use of a Tensor Empirical Interpolation Method (TEIM). The second phase consists in applying the Singular Value Decomposition (SVD) with low-rank truncation to the separate form of $f$ resulting from the first phase. Error estimates of the developed TEIM as well as truncated SVD decomposition are derived. Numerical experiments confirm that the decomposition techniques are efficient in terms of stability and accuracy.



%

\end{abstract}
\end{frontmatter}


\selectlanguage{english}
\section{Introduction and motivation}
\label{intro}

Low-rank tensor decompositions have recently become increasingly popular in numerous applications, as for example multivariate random and uncertainty quantification, parametrized problems, design analysis, image processing, high-order statistics, data analytics, etc (see \cite{tensor_D} for a review). 

In this paper, we are interested in tensor decomposition and/or approximate (low-rank) tensor decompositions of continuous multivariate functions. For simplicity purposes, we only deal with bivariate functions and shall introduce innovative algorithms in this context:
for a continuous function $f=f(x,y)$ defined over $I\times J$, we look for low-rank approximations $\tilde f^{(K)}(x,y)$ in the form 
\[
  \tilde f^{(K)}(x,y) = \sum_{k=1}^K \sigma_k\, \varphi_k(x)\,\psi_k(y)
\]
where $K$ is expected to be ``rather small'', $\sigma_k$ are some coefficients and
$\varphi_k$, $\psi_k$ are continuous univariate functions. \\


We focus on the decomposition of bivariate functions using the Empirical Interpolation~(EIM) methodology pioneered by Barrault et al. in~\cite{maday_cras} in 2004.
The EIM algorithm is a procedure that constructs interpolations of parametrized functions
and collocation points -- the so-called magic points -- within an incremental greedy procedure. The resulting interpolation operators are accurate
over all the parameter set by design. Theoretically, both a priori error bounds \cite{maday_generalisation} as well as a posteriori error bounds \cite{eim_aposteriori} of EIM have been derived.
 
%
Since the seminal papers, several extensions and variations of EIM have been developed.  Maday and Mula recently proposed a generalization of the EIM, the so-called GEIM~\cite{maday_general}. Their approach permits to approximately represent a given set as a linear combination of very few computable elements. In \cite{eim_weight}, Chen and al. extended the EIM to a weighted empirical interpolation method in order to approximate nonlinear parametric functions with weighted parameters. 
%
In \cite{eim_symetrique}, Casenave et al. developed a symmetric variant of~EIM to 
construct approximate representations of two-variable functions in separated form. 

In the present paper, from ideas in~\cite{eim_symetrique} we apply the EIM procedure ``direction-by-direction'' for each of the two variables, seeing the other variable as a parameter. This returns interpolation basis functions $q_k(x)$ and $s_\ell(y)$ for both variables $x$ and $y$ as well as a collocation grid made of the tensor product of the magic points~$(x_i,y_j)$. Then, the matrix $F=(f(x_i,y_j)_{i,j})$ is reduced using the Singular Value Decomposition (SVD). This provides a
new decomposition of $f$ and new component functions $\varphi_k(x)$ and $\psi_k(y)$ involving
a single summation over one index $k$.
It turns out that if the set of (nonnegative) singular values $\{\sigma_k\}_k$ decreases fast enough, then the truncation of the summation over the first principal components (where the 
$\sigma_k$ are arranged in decreasing order) provides an accurate approximate low-rank
decomposition of $f$.
%
%
~\\

This paper is outlined as follows. In section \ref{sec:teim} we develop the tensorized  empirical interpolation decomposition (TEIM) and we present an a priori error analysis of the method. To reduce the complexity of the TEIM, we apply the singular value decomposition (SVD)in section \ref{sec:svd} and we study the low-rank tensor decomposition by rank truncation.
We provide a theoretical error bound of the approximation.
Numerical results are then presented in section \ref{sec:num}. Finally, some conclusions and final remarks are given in section \ref{sec:conclusion}.
%
%
\section{Tensor Empirical Interpolation methodology}
\label{sec:teim}
Let $m$ and $n$ be two integers, $I$ and $J$ two closed intervals, $\Omega = I \times J\subset\mathbb{R}^2$ and $f :\Omega\longrightarrow \mathbb{R}$ a uniformly continuous function on $\Omega$. We look for an interpolation of the function $f$ in a separated form with respect to $x$ and $y$ with the following representation:
\begin{equation}
\mathcal{I}^{m,n} f(x,y) = \sum_{i=1}^m \sum_{j=1}^n f(x_i,y_j) \, q_i(x) \,s_j(y),
\quad (x,y)\in \Omega
\end{equation}
where $x_i$, $i\in\{1,...,m\}$ (resp. $y_j$, $j\in\{1,...,n\}$) are some interpolation points belonging to $I$ (resp. $J$). To build the interpolation operator $\mathcal{I}^{m,n}$, we need to construct, using a greedy  procedure, the set of spanning functions $\{q_1,q_2,...,q_m\}$ (resp. $\{s_1,s_2,...,s_n\}$) for the $x$ (resp. $y$) variable together with the associated interpolation points $\{x_1,x_2,...,x_m\}$ (resp. $\{y_1,y_2,...,y_n\}$).

\subsection{The TEIM algorithm}
%
The TEIM (Tensorized EIM) algorithm applies the EIM algorithm in each direction in order to return both basis functions and magic points.
It can be summarized in three steps as follows:
\begin{enumerate}
  \item First step: EIM in the $x$ direction. We consider $x$ as a space variable and $y$ is seen as a parameter living in $J$. The EIM algorithm then returns interpolation points $x_i$ and functions $q_i(x)$.  
  \begin{enumerate}
    \item We first set 
		\[
		\tilde{y}_1 = \arg\max_{y\in J} \|f(.,y)-0\|_{L^\infty_x(I)}, \quad
		x_1 = \arg\max_{x\in I} |f(x,\tilde{y}_1)-0|
		\]
		and 
		\[
		q^{(1)}_1(x) = \frac{f(x,\tilde{y}_1)}{f(x_1,\tilde{y}_1)}.
		\]
		With  $x_1$  and $q^{(1)}_1(x) $ we define the interpolation operator
		$\mathcal{I}^{(1)}_x$ by
		\[
		\mathcal{I}^{(1)}_x f(x,y)=f(x_1,y)\,q^{(1)}_1(x).
		\]
 \item For $k\in\{2,...,m\}$, we compute in a similar way  
\[
\tilde{y}_k = \arg\max_{y\in J} \|f(.,y)-\mathcal{I}^{(k-1)}_xf(.,y)\|_{L^\infty_x(I)},
\]
\[
x_k = \arg\max_{x\in I} |f(x,\tilde{y}_k)-\mathcal{I}^{(k-1)}_xf(x,y)|
\]
and 
\[
q^{(k)}_k(x)  = \frac{f(x,\tilde{y}_k)-\mathcal{I}^{(k-1)}_xf(x, \tilde{y}_k)}{f(x_k,\tilde{y}_k)-\mathcal{I}^{(k-1)}_x f(x_k,\tilde{y}_k)}.
\]
By construction, we have $q^{(k)}_k(x_k)=1$ and $q^{(k)}_k(x_i)=0$ for $i\in\{1,...,k-1\}$.
We slightly modify the original EIM algorithm by updating
the functions $q_i$ for all $i\in \{1,k-1\}$ in order to derive basis functions
that fulfill the Lagrange function property: $q^{(k)}_i(x_j)=\delta_{ij}$,
$i,j\in\{1,...,k-1\}$. For that we compute
\begin{equation}
q^{(k)}_i(x) = q^{(k-1)}_i(x) - q^{(k-1)}_i(x_k) q^{(k)}_k(x)
\quad \forall i\in \{1,k-1\}.
\end{equation}
At iterate $m$, using the $m$  points $\{x_1,...,x_m\}$ and the $m$ updated functions $\{q^{(m)}_i\}_{i=1,...,m}$, we get the rank-$m$ interpolation operator 
\begin{equation}\label{ix}
\mathcal{I}^{(m)}_x f(x,y)=\sum_{i=1}^{m} f(x_i,y)\, q^{(m)}_i(x),\quad 
y\in J.
\end{equation}
\end{enumerate}
%
%
%
\item Second step: we swap the role of $x$ and $y$. We suppose that $y$ is the space variable and $x$ is the parameter and we apply the previous procedure. We obtain a set of Lagrange functions $\{s^{(n)}_1,s^{(n)}_2,...,s^{(n)}_n\}$ of rank $n$ which satisfy the Lagrange property $s^{(n)}_i(y_j)=\delta_{ij}$, $\forall 1\leq i,j \leq n$ and the interpolation
\begin{equation}\label{iy}
  \mathcal{I}^{(n)}_y f(x,y) = \sum_{i=1}^n f(x,y_i) \, s^{(n)}_i(y),\quad x\in I. 
\end{equation} 

\item Third step: tensorized interpolation of the function $f$. By applying $\mathcal{I}^{(m)}_x$ defined by \eqref{ix} to the function $f$, for a fixed $y\in J$ we have
\[
\mathcal{I}^{(m)}_x f(x,y) = \sum_{i=1}^m f(x_i,y) \, q_i(x),\quad y\in J.
\]
Then by applying $\mathcal{I}^{(n)}_y$ defined by (\ref{iy}) to the function $\mathcal{I}^{(m)}_x f$ we obtain
\begin{equation}
\mathcal{I}^{(n)}_y\mathcal{I}^{(m)}_x f(x,y) =
\sum_{i=1}^m \sum_{j=1}^n f(x_i,y_j) \, q^{(m)}_i(x) s^{(n)}_j(y)
= \mathcal{I}^{(m)}_x \mathcal{I}^{(n)}_y f(x,y).
\label{eq:1}
\end{equation}

The interpolation operator $\mathcal{I}^{m,n}$ of $f$ is then naturally defined by:

\begin{equation}\label{tensor_decomp}
\mathcal{I}^{m,n} f(x,y) = \mathcal{I}^{(n)}_y\mathcal{I}^{(m)}_x f(x,y) 
= \mathcal{I}^{(m)}_x \mathcal{I}^{(n)}_y f(x,y)
= \sum_{i=1}^m \sum_{j=1}^n f(x_i,y_j) \, q^{(m)}_i(x) s^{(n)}_j(y).  
\end{equation}
\end{enumerate}

\subsection{Properties of the TEIM method}
We now summarize properties of the TEIM algorithm and detail proofs.\\

\begin{proposition}\label{propo_delta}
The matrix $Q^m$ (resp. $S^n$ ) defined by~: $Q^m_{i,j}= q^{(m)}_j(x_i)$, $\forall 1\leq i,j \leq m $ (resp. $S^n_{i,j}= s^{(n)}_j(x_i)$, $\forall 1\leq i,j \leq n $) is equal to the identity matrix of order $m\times m$  (resp. of order $n \times n$).
\end{proposition}

\begin{proof}
We proceed by induction. Clearly, $Q^1_{1,1}=q^{(1)}_1(x_1)=1$. Now, suppose that $Q^{m-1} = I_{m-1}$ ( where $I_{m-1}$ is the identity matrix of order $(m-1)\times (m-1)$). By construction, $q^{(m)}_m(x_m)=1$ and for all $j< m$, 
\begin{equation}\label{qM}
  q^{(m)}_m(x_j)= \frac{f(x_j,\tilde{y}_m)-\mathcal{I}^{(m-1)}_xf(x_j, \tilde{y}_m)}{f(x_m,y_m)-\mathcal{I}^{(m-1)}_x f(x_m,\tilde{y}_m)}
\end{equation} 
where,
\[
  \mathcal{I}^{(m-1)}_x f(x_j,\tilde{y}_m) = \sum_{i=1}^{m-1} f(x_i,\tilde{y}_m)\,q^{(m-1)}_i(x_j).
\]
Using the induction hypothesis we get $\forall j<m$~: $q_i^{(m-1)}(x_j) = \delta_{ij}$ and then 
\[
\mathcal{I}^{(m-1)}_x f(x_j,\tilde{y}_m)=f(x_j,\tilde{y}_m). 
\]
We replace in (\ref{qM}) then $\forall j<m$, $q^{(m)}_m(x_j)=0$. $\forall j<m$, $q^{(m)}_j(x_m)=q^{(m-1)}_j(x_m)-q^{(m-1)}_j(x_m)q^{(m)}_m(x_m)=0$. Furthermore, for all $\{i,j\}<m$, $q^{(m)}_j(x_i)=q^{(m-1)}_j(x_i)-q^{(m-1)}_j(x_i)\,q^{(m)}_m(x_i)=\delta_{ij}$ (induction hypothesis).
The proof for the matrix $S^{n}$ can be done in the same way.
\qed\\
\end{proof}

We have also the following interesting result:

\begin{proposition}\label{prop_interp_prop}
The interpolation operator $\mathcal{I}^{m,n}$ of $f$ satisfies the following interpolation property
\begin{equation}
\left\{
  \begin{array}{ll}
   \mathcal{I}^{m,n} f(x_k,y) &=  f(x_k,y) \quad \forall y\in J,    \\
   \mathcal{I}^{m,n} f(x,y_k) &=  f(x,y_k) \quad \forall x\in I.
  \end{array}
\right.
\label{eq:reinforced}
\end{equation}
\end{proposition}

\begin{proof} Straightforwardly, from the Lagrange property, we have
\[
\mathcal{I}^{m,n} f(x_k,y)=\sum_{i=1}^m \sum_{j=1}^n f(x_i,y_j) \, q_i^{(m)}(x_k)\, s_j^{(n)}(y) = \sum_{j=1}^n f(x_k,y_j)\, s_j^{(n)}(y) =f(x_k,y), 
\]
for all $y\in J$ and 
\[
\mathcal{I}^{m,n} f(x,y_k)=\sum_{i=1}^m \sum_{j=1}^n f(x_i,y_j) \, q_i^{(m)}(x)\, 
s_j^{(n)}(y_k) = \sum_{i=1}^m f(x_i,y_k)\, q_i^{(m)}(x) =f(x,y_k)
\]
for all $x\in I$.
\qed
\end{proof}

\begin{remark}
Property~\eqref{eq:reinforced} can be seen as a reinforced interpolation
condition because interpolation conditions are not only valid on collocation points but also on an array of intervals. Notice that Casenave et al.~\cite{sym_eim} also have reinforced interpolation properties with their symmetric EIM approach.
%
\end{remark}
  
\subsection{A priori error estimate for TEIM} 
\label{sec:error}
We  follow the same ideas as in \cite{maday_cras} and \cite{maday_generalisation}.  
We define the Lebesgue constants $L_m$ and $\tilde{L}_n$ (see~\cite{lambda_poly}) by 
\[
L_m=\sup_{x\in I}\ \sum^m_{i=1}\vert q_i^{(m)}(x) \vert, \quad
\tilde{L}_n=\sup_{y\in J}\ \sum^n_{i=1}\vert s_i^{(n)}(y) \vert
\] 
\begin{proposition}\label{prop_erreur}
The interpolation error $\varepsilon^{m,n} \equiv \Vert f-I^{m,n}f\Vert_{L^{\infty}}$ satisfies 
\begin{equation}\label{error}
 \varepsilon^{m,n} \leq \varepsilon^*_{m,n}\left(1+L_m \tilde{L}_n\right), 
\end{equation}
where  $\varepsilon^*_{m,n} \equiv \inf_{f^*\in W^{m,n}} \|f - f^*\|_{L^\infty} $ and $f^*$ is the best approximation of $f$ in 
\[
W^{m,n}=\mathop{span}(q_k(x) s_\ell(y),\ 1\leq k\leq m,\ 1\leq \ell\leq n).
\]
The Lebesgue constants $L_m$ and $\tilde{L}_n$ verify 
\begin{equation}\label{lambda}
  L_m \leq 2^{m}-1, \quad  \tilde{L}_n \leq 2^{n}-1.
\end{equation}
\end{proposition}

\begin{proof} 
For all $x$ and $y$,
\begin{align*}
  \mathcal{I}^{m,n}f(x,y)-f^*_{m,n}(x,y)&=\sum^m_{i=1}\sum^n_{j=1}\left(f(x_i,y_j)-f^*_{m,n}(x_i,y_j)\right) q^{(m)}_i(x)s^{(n)}_j(y) \\&:= \sum^m_{i=1}\sum^n_{j=1} e^*_{m,n}(x_i,y_j)q^{(m)}_i(x)s^{(n)}_j(y).
\end{align*}
As a result, we have the following estimation of the error
\begin{align*}
  \Vert f(x,y)-\mathcal{I}^{m,n}f(x,y)\Vert_{L^{\infty}} &\leq \Vert f(x,y)-f^*_{m,n}(x,y)\Vert_{L^{\infty}} + \Vert \mathcal{I}^{m,n}f(x,y)-f^*_{m,n}(x,y)\Vert_{L^{\infty}}%
\\&\leq \varepsilon^*_{m,n} + \Vert \sum^m_{i=1}\sum^n_{j=1} e^*_{m,n}(x_i,y_j)q^{(m)}_i(x)s^{(n)}_j(y)\Vert_{L^{\infty}}%
\\& \leq \varepsilon^*_{m,n} + \max_{i\in\{1,m\}}\left(\max_{j\in\{1,n\}} |e^*_{m,n}(x_i,y_j)|\right) L_m \tilde{L}_n%
\\ &\leq \left(1+ L_m \tilde{L}_n\right)\varepsilon^*_{m,n}.
\end{align*}
The estimation (\ref{error}) is then proven. For the formula (\ref{lambda}) we begin by proving, by induction, that
\begin{equation}\label{lambda_qm}
  \vert q^{(m)}_k(x)\vert \leq 2^{m-1}, \quad \forall 1\leq k\leq m.
\end{equation}
We have $\vert q^{(1)}_k(x) \vert \leq 1:=2^{1-1}$. Suppose now that (\ref{lambda_qm}) is true for iterate $(m-1)$ then
\begin{align*}
  \vert q^{(m)}_k(x) \vert &= \vert q^{(m-1)}_k(x) - q^{(m-1)}_k(x_m)q^{(m)}_m(x) \vert%
 \leq 2^{m-2}+2^{m-2} = 2^{m-1}. 
\end{align*}
To finish the proof of (\ref{lambda}) it is sufficient to note that 
\begin{align*}
 \vert L_m \vert &= \sup_{x\in I}\  \sum^m_{k=1} \vert q^{(m)}_k (x)\vert%
\leq \sum^m_{k=1} \vert q^{(m)}_k (x)\vert %
\leq \sum^{m-1}_{k=0}  2^k = 2^m-1.
\end{align*}
The estimation of $\tilde{L}_n$ can be done in the same way.
\qed\\
\end{proof}
\begin{remark}
According to Proposition \ref{prop_erreur}, the method is efficient as soon as
$2^{m+n}\times \varepsilon^*_{m,n}$ is small enough for given $m$ and $n$.
It occurs that the upper bound~\eqref{lambda_qm} is pessimistic. Numerical results
at the end of this paper show that the Lebesgue constants are observed to behave linearly
with respect to $m$ or $n$. \\
\end{remark}
%

%
The tensor decomposition of the function $f$ given by (\ref{tensor_decomp}) leads to a complexity of ($m\times n$)  products of one-variable functions. The reduction of this complexity will be investigated and discussed in the next section.
%
\section{Reduction of complexity: SVD decomposition}
\label{sec:svd}
Let $\bm{q}(x)\in\mathbb{R}^m$ (resp. $\bm{s}(y)\in\mathbb{R}^n$) the vector defined by~: $\bm{q}(x)=(q^{(m)}_1(x),q^{(m)}_2(x),...,q^{(m)}_m(x))^T$ (resp. $\bm{s}(y)=(s^{(n)}_1(y),s^{(n)}_2(y),...,s^{(n)}_n(y))^T$). The matrix $F\in\mathscr{M}_{mn}(\mathbb{R})$ is given by $F_{ij}=f(x_i,y_j)$, $1\leq i\leq m$, $1\leq j\leq n$. Then (\ref{tensor_decomp}) can be written as follows

\begin{equation}\label{imn_svd1}
\mathcal{I}^{m,n} f(x,y) = \bm{s}(y)^T\, F^T\, \bm{q}(x)
= \bm{q}(x)^T\, F\, \bm{s}(y). 
\end{equation}

According to the SVD decomposition theorem (\cite{svd,book_svd}) the matrix 
$F\in\mathscr{M}_{mn}(\mathbb{R})$ can be decomposed into 
\begin{equation}
 F = U \,\Sigma\, V^T,
\end{equation}
where $U\in\mathscr{M}_m(\mathbb{R})$ (resp. $V\in\mathscr{M}_n(\mathbb{R})$) is a unitary matrix, $U^T U=I_m$ (resp. $V^T V=I_n$) and $\Sigma\in\mathscr{M}_{mn}(\mathbb{R})$, such that $\Sigma_{ij}=0$ for $i\neq j$ and $\Sigma_{ii}:=\sigma_i$, $\sigma_1\geq \sigma_2\geq ...\geq \sigma_{\min(m,n)}\geq 0$. Using the SVD decomposition, expression~(\ref{imn_svd1}) becomes
\[
\mathcal{I}^{m,n} f(x,y) = \bm{s}(y)^T V \,\Sigma^T\, U^T \bm{q}(x)
= (V^T\bm{s}(y))^T  \,\Sigma^T\, (U^T \bm{q}(x)).
\]
Let $\bm{\Phi}(x)=U^T \bm{q}(x)$ and $\bm{\Psi}(y)=V^T \bm{s}(y)$ then

\[
\mathcal{I}^{m,n} f(x,y) = (\bm{\Psi}(y))^T\, \Sigma^T\, \bm{\Phi}(x).
\]
Denoting $\varphi_k(x)$ (resp. $\psi_k(y)$) the $k$-th component of $\bm{\Phi}(x)$
(resp. $\bm{\Psi}(y)$), the latter expression can be written
\begin{equation}\label{rank_k}
\mathcal{I}^{m,n} f(x,y) = \sum_{k=1}^{\min(m,n)} \sigma_k\, 
\varphi_k(x)\, \psi_k(y).
\end{equation}
If the singular value sequence  $\{\sigma_k\}_k$ rapidly decays, the singular decomposition of $F$ can be reasonably truncated to a rank $K\leq \min(m,n)$.
We then consider the following reduced tensor representation of~$f$:
\begin{equation}
  \tilde f^{(K)}(x,y) = \sum_{k=1}^K \sigma_k\, \varphi_k(x)\,\psi_k(y),
	\label{eq:tildef}
\end{equation}
\subsection{Error analysis of the truncated TEIM-SVD decomposition}

\begin{proposition}\label{prop31}
We have the following error estimation:
\begin{equation}
\|f-\tilde f^{(K)}\|_{L^\infty} \leq 
(1+L_m \tilde{L}_n) \
\inf_{f^*\in W^{m,n}} \|f - f^*\|_{L^\infty} \ + \ 
L_m \tilde L_n \sqrt{mn}  \left(\dfrac{\displaystyle{\sum_{k=K+1}^n \lambda^k}}{\displaystyle{\sum_{k=1}^n \lambda^k}}\right)^{1/2}\, \|F\|_F,
\label{eq:errork}
\end{equation}
where $\|.\|_F$ denotes the standard Frobenius matrix norm and $\lambda^k=(\sigma^k)^2$ are the (nonnegative) eigenvalues of the symmetric matrix $(F^T F)$.\\
\end{proposition}
\begin{proof}[proposition~\ref{prop31}]
Let us denote by $\bm{r}_k\in\mathbb{R}^m$ the $k$-th column vector of $U$ and by $\pi^{(K)}$ the
orthogonal projector on $W^{(K)}:=\mathop{span}(\bm{r}_1,...,\bm{r}_K)$ with $K < \min(m,n)$.
Let us denote by $(\bm{f}_k)_{k\in\{1,...,n\}}$ the $n$ column vectors of matrix $F$.
Denoting 
\[
\tilde \Sigma^{(K)} = \begin{pmatrix} I_K & (0) \\ (0) & (0)\end{pmatrix} \Sigma
\]
and $\tilde F^{(K)} = U\, \tilde \Sigma^{(K)}\, V^T$, we have
$\sum_{k=1}^K \|\bm{f}_k - \pi^{(K)} \bm{f}_k\|^2_{2,\mathbb{R}^m} = \|F-\tilde F^{(K)}\|_F^2$.
From standard SVD decomposition error estimations (\cite{book_svd}), we have
the following error estimate
\[
\|F-\tilde F^{(K)}\|_F^2=
\sum_{k=1}^K \|\bm{f}_k - \pi^{(K)} \bm{f}_k\|^2_{2,\mathbb{R}^m} \leq 
\dfrac{\displaystyle{\sum_{k=K+1}^n \lambda^k}}{\displaystyle{\sum_{k=1}^n \lambda^k}}\ \sum_{k=1}^n \|\bm{f}_k\|^2_{2,\mathbb{R}^m}
= \dfrac{\displaystyle{\sum_{k=K+1}^n \lambda^k}}{\displaystyle{\sum_{k=1}^n \lambda^k}}\, \|F\|_F^2.
\]
By the triangular inequality we have
\[
\|f - \tilde f^{(K)}\|_{L^\infty} \leq 
  \| f - \mathcal{I}^{m,n}_{} f \|_{L^\infty} 
+ \| \mathcal{I}^{m,n}_{} f - \mathcal{I}^{m,n}_{} (\tilde f^{(K)}) \|_{L^\infty} 
+ \| \mathcal{I}^{m,n}_{} f^{(K)} - \tilde f^{(K)} \|_{L^\infty} 
\]
From the previous result of proposition~\ref{prop_erreur}, it has been show that
\[
\| f - \mathcal{I}^{m,n}_{} f \|_{L^\infty}  \leq (1+L_m \tilde{L}_n) \
\inf_{f^*\in W^{m,n}} \|f - f^*\|_{L^\infty}.
\]
Because $\tilde f^{(K)}\in\mathop{span}(\{\varphi_k(x)\psi_k(y)\}_k)$, we have also
$\tilde f^{(K)}\in W^{m,n}$. Thus,
$\tilde f^{(K)}=\mathcal{I}^{m,n} \tilde f^{(K)}$ and
\[
\|f - \tilde f^{(K)}\|_{L^\infty} \leq 
 (1+L_m\tilde{L}_n) 
\inf_{f^*\in W^{m,n}} \|f - f^*\|_{L^\infty}
\ + \ \| \mathcal{I}^{m,n}_{} f - \mathcal{I}^{m,n}_{} (\tilde f^{(K)}) \|_{L^\infty}.
\]
Functions $\mathcal{I}^{m,n}_{} f$ and $\mathcal{I}^{m,n}_{} (\tilde f^{(K)})$ are respectively given by
\[
\mathcal{I}^{m,n} f(x,y) = \sum_{i,j} f(x_i,y_j)\, q^{(m)}_i(x)\, s^{(n)}_j(y), \ \quad
\mathcal{I}^{m,n} \tilde f^{(K)}(x,y) = \sum_{i,j} \tilde f^{(K)}(x_i,y_j)\, q^{(m)}_i(x)\, s^{(n)}_j(y), 
\]
hence
\[
(\mathcal{I}^{m,n} f - \mathcal{I}^{m,n} \tilde f^{(K)})(x,y)
= \sum_{i=1}^m\sum_{j=1}^n (f(x_i,y_j) - \tilde f^{(K)}(x_i,y_j))  \, q^{(m)}_i(x) \, s^{(n)}_j(y).
\]
and
\begin{eqnarray*}
\|\mathcal{I}^{m,n} f - \mathcal{I}^{m,n} \tilde f^{(K)}\|_{L^\infty}
&\leq & \sum_{i=1}^m\sum_{j=1}^n |(f(x_i,y_j) - \tilde f^{(K)}(x_i,y_j))  |
  \ L_m \tilde L_n\\
	&\leq & L_m \tilde L_n \sqrt{mn} \left(\sum_{i=1}^m\sum_{j=1}^n |(f(x_i,y_j) - \tilde f^{(K)}(x_i,y_j))  |^2\right)^{1/2} \\
	&=&  L_m \tilde L_n \sqrt{mn}  \, \|F-\tilde F^{(K)}\|_F.
\end{eqnarray*}
This completes the proof. \qed\\
\end{proof}
\section{Numerical experiments}
\label{sec:num}

In this section,  the tensor empirical interpolation method TEIM and 
its approximate truncated SVD decomposition are tested on the continuous function
\begin{equation}\label{funct_test}
f(x,y) =  x+y+xy+e^{-(x^2+y^2)}+\sin(3\pi y)-\sin(\pi x y^2+\pi x e^{-y})  
\end{equation}
defined on $\Omega=I\times J=[0,1]^2$ (plotted on Fig.~\ref{f_fappro}, left plot).
First, we demonstrate the accuracy and efficiency of the TEIM. We choose $m=n=10$. The couple of points $(x_i,y_j)$, $1\leq i,j\leq 10$ selected by the TEIM algorithm are represented in Fig.~\ref{fig:funct_pointInterpo}.
\begin{figure}[h!]
\begin{center}
\includegraphics[height=6.5cm]{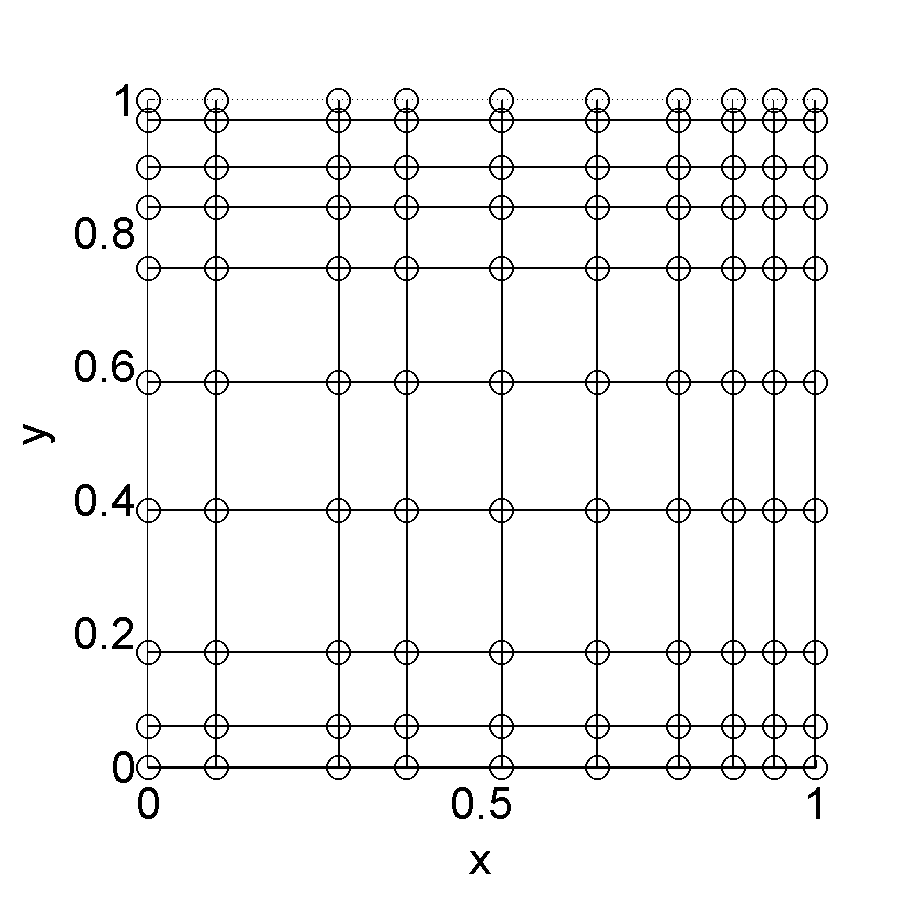}
\end{center}
\caption{Distribution of the first $100$ points ($x_i$,$y_j$), $1\leq \{i,j\} \leq 10$ selected by  the TEIM algorithm.}
\label{fig:funct_pointInterpo}
\end{figure}
%
%
The four first basis functions $q_i(x)$ (resp. $s_j(y)$) resulting from the TEIM as well as  the ten interpolation points $x_i$ (resp. $y_j$) are represented in Fig.~\ref{fig:base_q} (resp. Fig.~\ref{fig:base_s}). As it was proved in section \ref{sec:teim}, the basis functions $q_i$  (resp. $s_j$) satisfy the property $q_i(x_j)=\delta_{ij}$ for all $i \in \{1,10\}$ (resp. $s_j(y_i)=\delta_{ji}$ for all $j \in \{1,10\}$). We can also observe that $\|q_k\|_{L^\infty_x}$ and 
$\|s_\ell\|_{L^\infty_y}$ are not far from $1$.
\begin{figure}[h!]
\begin{center}
\includegraphics[width=0.85\textwidth]{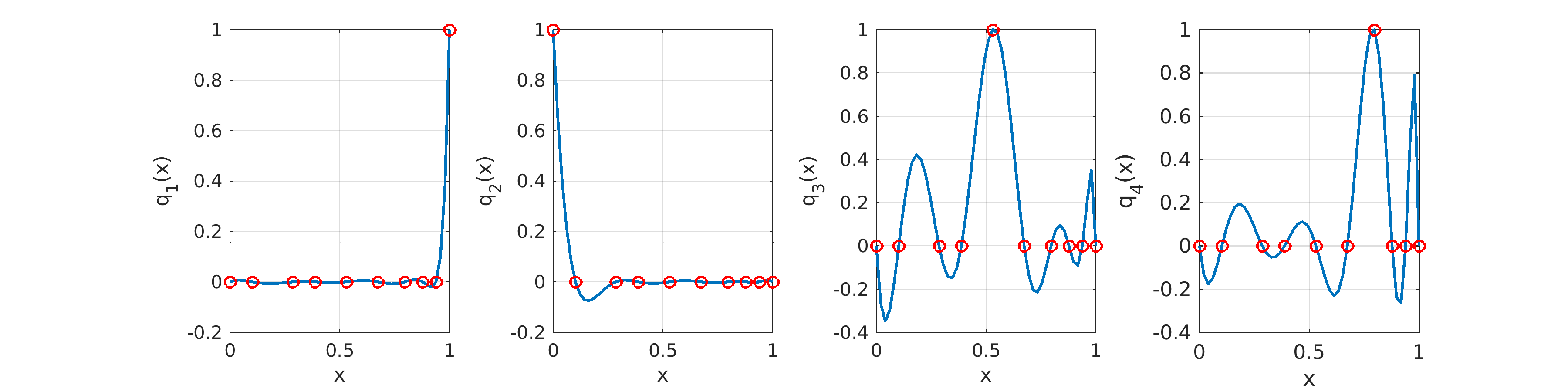}
\end{center}
\caption{The four first basis functions $q_i(x)$ with the ten first TEIM points $x_i$.}
\label{fig:base_q}
\end{figure}
\begin{figure}[h!]
\begin{center}
\includegraphics[width=0.85\textwidth]{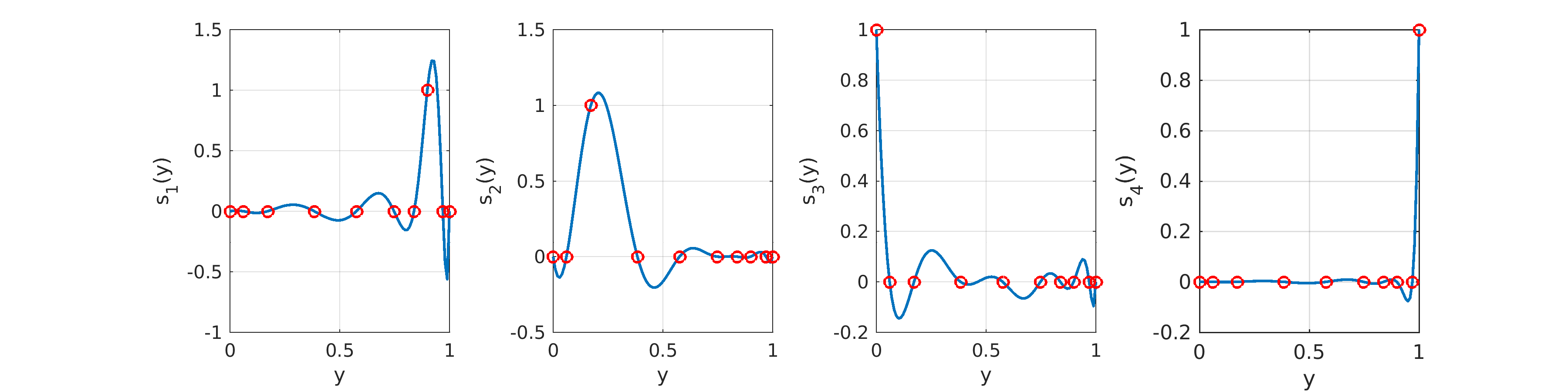}
\end{center}
\caption{The four first basis functions $s_j(y)$ with the ten first TEIM points $y_j$.}
\label{fig:base_s}
\end{figure}
We present in Fig.~\ref{fig:lambdas} the 
computed values of $L_m$ and~$\tilde{L}_n$ calculated for different number of iterations $m$ and $n$. We observe that $L_m$ and~$\tilde{L}_n$  linearly grow
with $m$ and $n$ respectively and they provide a very stable interpolation provided by TEIM. The upper bounds of the Lebesgue constants $L_m$ and $\tilde{L}_n$ derived in section \ref{sec:error} appear to be pessimistic.   

To numerically study the TEIM error, we chosse $m=n$ and we compute the infinite norm of the error $\Vert f -\mathcal{I}^{m,n}_{x,y}\Vert_{L^{\infty}}$ for various values of $m$ from $1$ to $10$. The TEIM error is represented in Fig.~\ref{fig:erreur} (right) in semi-logarithmic scale. Note that the error converges rapidly with $m$ which confirms the accuracy and efficiency of the developed TEIM
approach.
\begin{figure}[h!]
\begin{center}
\includegraphics[height=7.5cm]{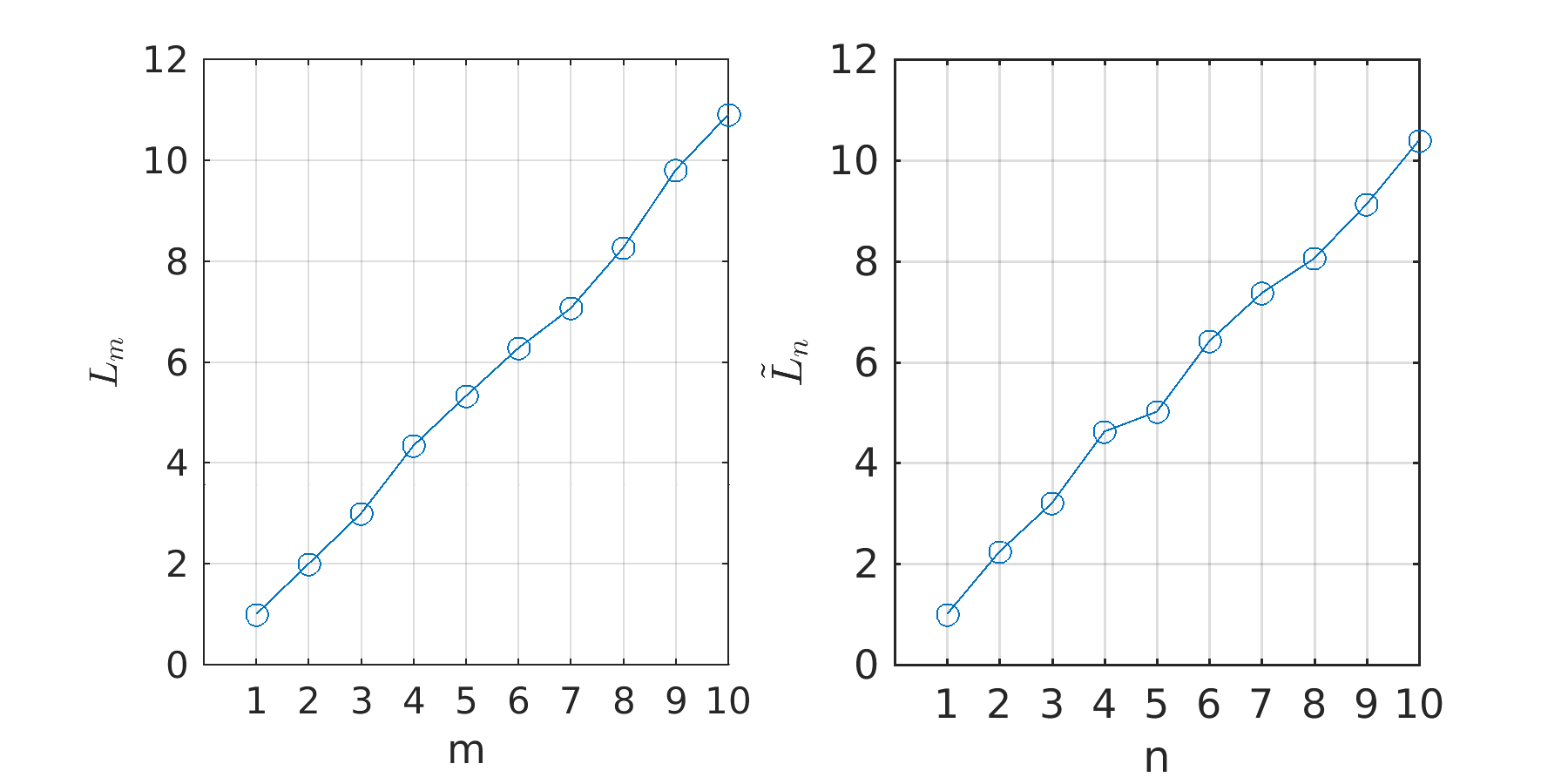}
\end{center}
\caption{The Lebesgue constants $L_m$ (left) and $\tilde{L}_n$ (right) with respect
to $m$ (resp. $n$), $\{m,n\} \in [1,10]$.}
\label{fig:lambdas}
\end{figure}
\begin{figure}[h!]
\begin{center}
\includegraphics[height=7.5cm]{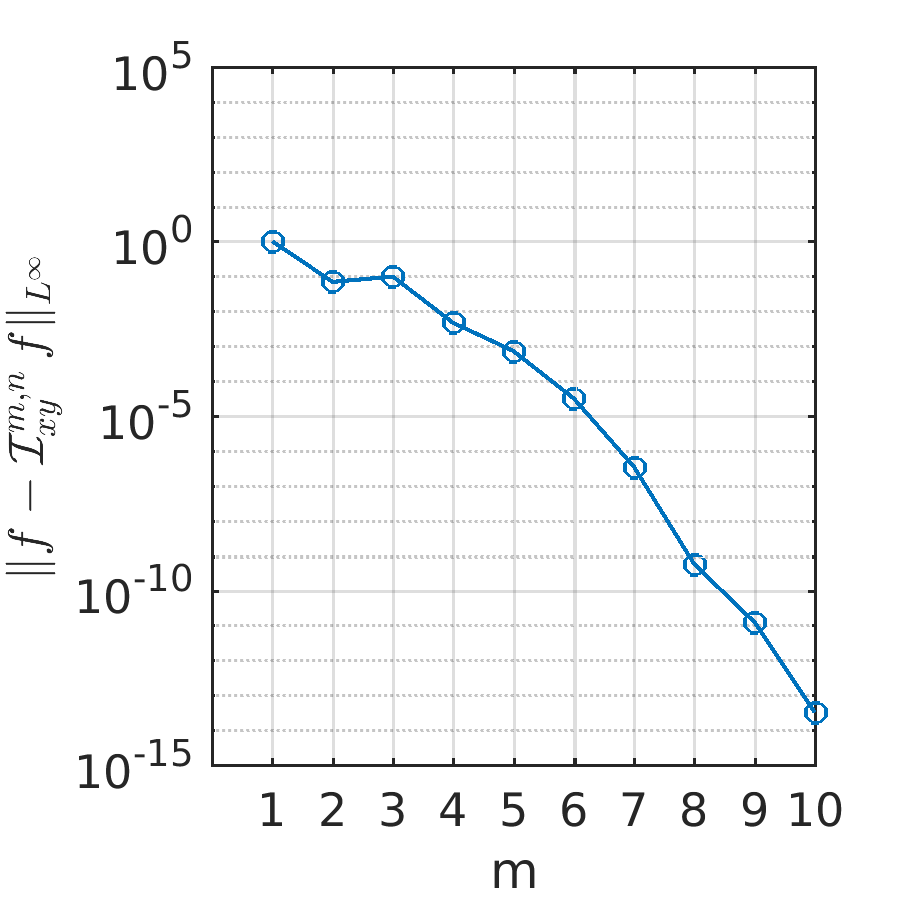}
\end{center}
\caption{Convergence of the TEIM interpolation approach with respect to $m$ ($n=m$).}
\label{fig:erreur}
\end{figure}

In order to reduce the complexity induced by TEIM, we apply the SVD to the interpolation operator of~$f$ (expression~\ref{tensor_decomp}) to get an approximation of~$f$ in the form (\ref{rank_k}) as it has been explained in section \ref{sec:svd}. The four first functions $\varphi(x)$ (resp. $\psi(y)$) resulting from the SVD decomposition applied to the operator~$\mathcal{I}^{10,10}$ are represented in Fig.~\ref{fig:phi} (resp. Fig.~\ref{fig:psi}). Fig.~\ref{fig:sigma} shows the $10$ singular values $\sigma_k$, $1\leq k\leq 10$ in a semi-logarithmic scale. We observe that the values $\sigma_k$ decay rapidly. Then, the singular decomposition of~$F$ can be indeed reduced to a truncated low-rank. In Fig.~\ref{f_fappro}, we can see that with the only rank-2 approximation 
\[
\tilde f^{(2)}(x,y) = \sigma_1\, \varphi_1(x)\,\psi_1(y)
+\sigma_2\, \varphi_2(x)\,\psi_2(y),
\]
we have a good reduced representation of the function $f$, with infinite-norm relative error of order 1\%.

\begin{figure}[h!]
\begin{center}
\includegraphics[height=4.cm]{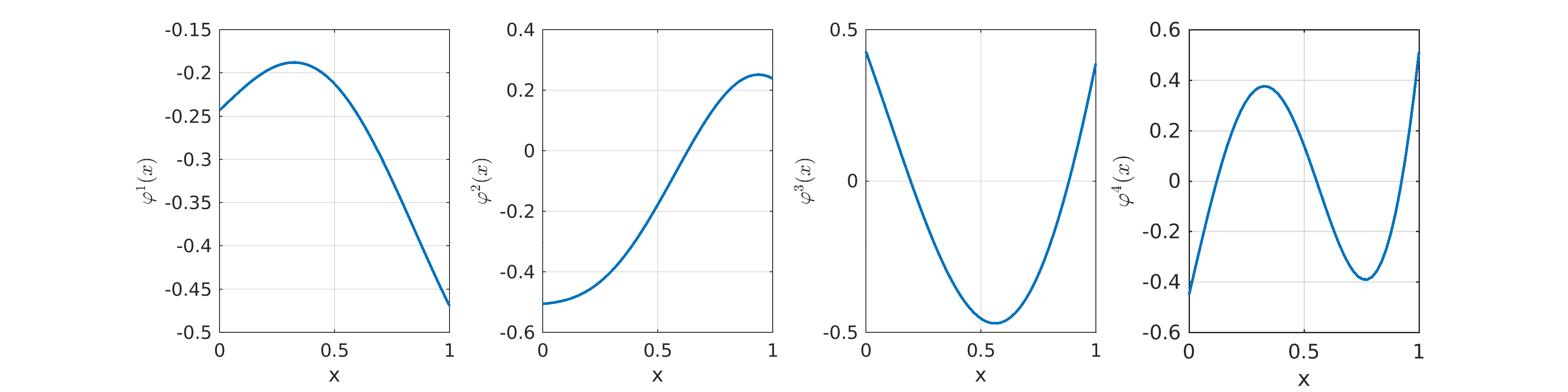}
\end{center}
\caption{The four first functions $\varphi(x)$ resulting from the SVD decomposition.}
\label{fig:phi}
\end{figure}

\begin{figure}[h!]
\begin{center}
\includegraphics[height=4.cm]{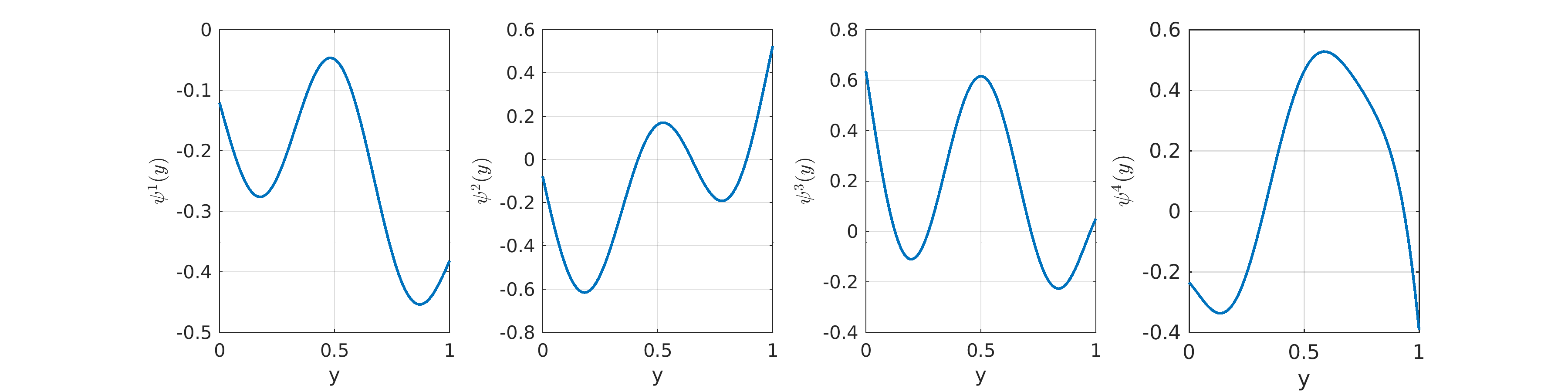}
\end{center}
\caption{The four first functions $\psi(y)$ resulting from the SVD decomposition.}
\label{fig:psi}
\end{figure}

\begin{figure}[h!]
\begin{center}
\includegraphics[height=6.5cm]{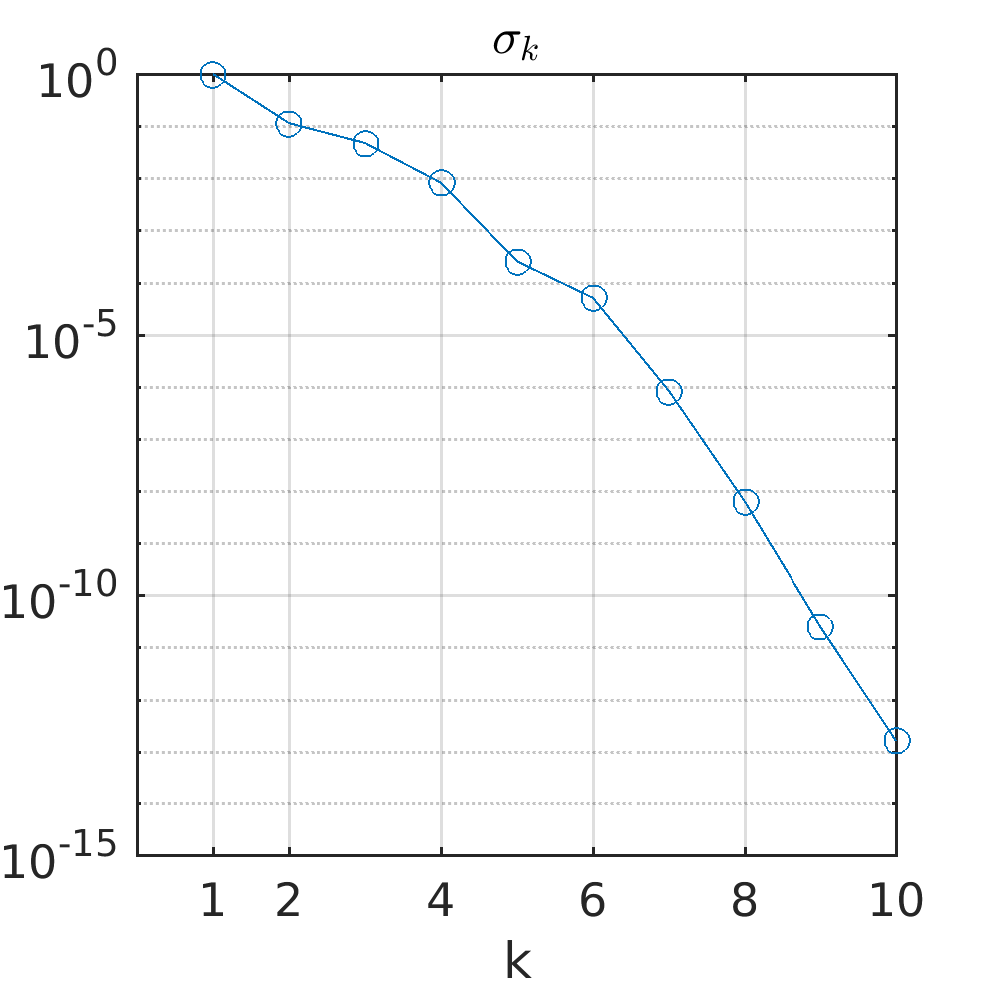}
\end{center}
\caption{The singular values $\sigma_k$ from the SVD decomposition.}
\label{fig:sigma}
\end{figure}

\begin{figure}[h!]
\begin{center}
\includegraphics[height=6.5cm]{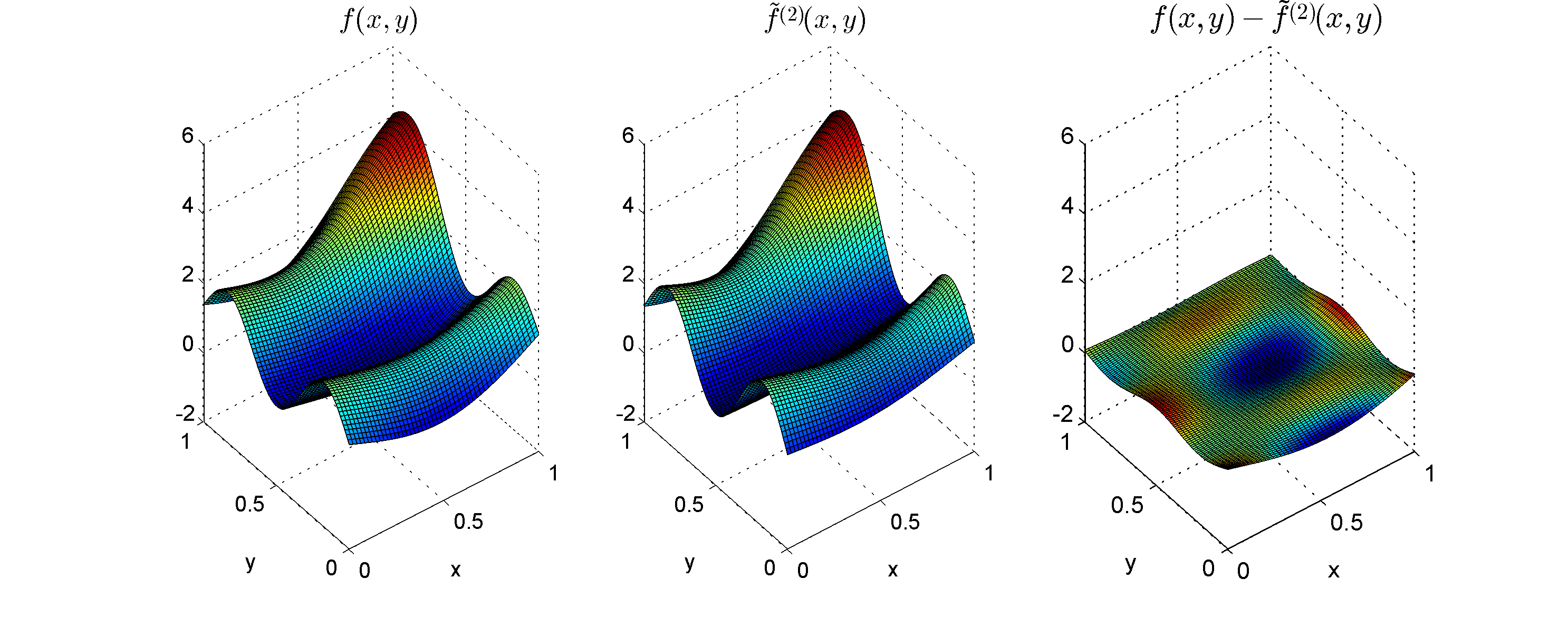}
\end{center}
\caption{The tested function $f$ given analytically in expression~(\ref{funct_test}) (left). The reduced representation of $f$ (middle).  The difference between the function $f$ and its approximation $\tilde{f}^2$ (right).}
\label{f_fappro}
\end{figure}


\section{Concluding remarks and perspectives}
\label{sec:conclusion}
In this work, we have derived a Tensor Empirical Interpolation Method (TEIM) for bivariate functions. The presented method is inspired by the classical EIM where the greedy procedure is used to compute the interpolation points and the basis functions for each direction. The algorithm returns interpolation functions that
directly have the Lagrange property.
The TEIM thus provides a natural interpolation of the bivariate function over a tensorized collocation grid. It leads to a complexity of ($m\times n$)  products of one-variable functions. To reduce this complexity, we apply the Singular Value Decomposition (SVD) to the TEIM separate variable decomposition.  
If the singular values decay fast enough, one can additionally truncate the
decomposition, leading to a low-rank representation. We have performed an
error estimate in this case, combining TEIM interpolation error and SVD truncation error.
Numerical experiments confirm that the mixed SVD-TEIM decomposition have a very good behavior in terms of stability and accuracy. \\

This paper only deals with bi-variate, but there is no difficulty to extend the 
proposed approach to the multivariate case, at least in the ``low-dimensional'' case
(say a number of parameters less than 6). For SVD decomposition in this case, one  can use extensions of SVD decomposition for tensors. Future works will address
this issue. For perspective, we also plan to apply the decompositions developed 
in this paper for achieving non-intrusive reduced-order modeling of time-dependent solutions of partial differential equations \cite{audouze}, with separate decomposition of space and time variables.

%

\end{document}